\documentclass[12pt]{amsart}
\usepackage{geometry} 
\usepackage{graphicx}
\usepackage{amsfonts}
\usepackage{amsthm}
\usepackage{amsmath}
\usepackage{amssymb}
\input cyracc.def
\DeclareFontFamily{U}{russian}{}
\DeclareFontShape{U}{russian}{m}{n}
        { <5><6> wncyr5
          <7><8><9> wncyr7
          <10><10.95><12><14.4><17.28><20.74><24.88> wncyr10 }{}
\DeclareSymbolFont{Russian}{U}{russian}{m}{n}
\DeclareSymbolFontAlphabet{\mathcyr}{Russian}
\makeatletter
\let\@math@cyr\mathcyr
\renewcommand{\mathcyr}[1]{\@math@cyr{\cyracc #1}}
\makeatother
\newcommand{\sha}{\mathcyr{SH}}
\newtheorem{theorem}{Theorem}

\newtheorem{lemma}[theorem]{Lemma}
\newtheorem{proposition}[theorem]{Proposition}
\theoremstyle{definition}

\newtheorem{remark}[theorem]{Remark}
\newcommand{\co}{\colon\,}
\newcommand{\bA}{\mathbb A}
\newcommand{\bT}{\mathbb T}
\newcommand{\bR}{\mathbb R}
\newcommand{\bC}{\mathbb C}
\newcommand{\bG}{\mathbb G}
\newcommand{\bK}{\mathbb K}
\newcommand{\bZ}{\mathbb Z}
\newcommand{\bP}{\mathbb P}
\newcommand{\bH}{\mathbb H}

\newcommand{\bS}{\mathbb S}
\newcommand{\cA}{\mathcal A}
\newcommand{\cD}{\mathcal D}
\newcommand{\cE}{\mathcal E}
\newcommand{\cF}{\mathcal F}

\newcommand{\cI}{\mathcal I}

\newcommand{\cL}{\mathcal L}
\newcommand{\cO}{\mathcal O}
\newcommand{\cP}{\mathcal P}

\newcommand{\cV}{\mathcal V}

\newcommand{\tK}{\widetilde K}

\newcommand{\tI}{\widetilde{\mathrm I}}
\newcommand{\tIA}{\widetilde{\mathrm{IA}}}
\newcommand{\pt}{\text{pt}}
\newcommand{\lp}{\textup{(}}
\newcommand{\rp}{\textup{)}}

\newcommand{\Gal}{\operatorname{Gal}}

\newcommand{\Br}{\operatorname{Br}}

\newcommand{\sn}{\operatorname{sn}}

\newcommand{\Spec}{\operatorname{Spec}}
\newcommand{\Pic}{\operatorname{Pic}}

\newcommand{\id}{\text{id}}
\newcommand{\ts}{\widetilde{\sigma}}

\renewcommand{\top}{^\text{top}}
\title[Algebraic T-duality]{Algebraic $K$-theory and derived\\
equivalences suggested by\\ T-duality for torus orientifolds}

\author{Jonathan Rosenberg}
\address{Department of Mathematics\\
University of Maryland\\
College Park, MD 20742-4015, USA} 
\email[Jonathan Rosenberg]{jmr@math.umd.edu}
\thanks{Partially supported by NSF grant DMS-1206159.}
\dedicatory{In honor of the birthday of my good friend Chuck Weibel}
  
\begin{document}

\begin{abstract}
  We show that certain isomorphisms of (twisted)
  $KR$-groups that underlie T-dualities of torus
  orientifold string theories have purely algebraic
  analogues in terms of algebraic $K$-theory of real varieties and
  equivalences of derived categories of (twisted) coherent
  sheaves. The most interesting conclusion is a kind of Mukai duality
  in which the ``dual abelian variety'' to a smooth
  projective genus-$1$ curve over $\bR$ with no real points is
  (mildly) noncommutative. 
\end{abstract}
\keywords{orientifold, $KR$-theory, elliptic curve, twisting, Mukai duality,
  T-duality, derived   category, coherent sheaves, algebraic
  $K$-theory, real algebraic geometry, Azumaya algebra} 
\subjclass[2010]{Primary 14F05; Secondary 19E08 19L64  81T30  14F22  
  14H60  14H81.} 

\maketitle

\section{Introduction}
\label{sec:intro}

\subsection{Background}
\label{sec:background}
This paper grew out of joint work with Charles Doran and Stefan
Mendez-Diez \cite{Doran:2013sxa,MR3316647}, in which we studied type II
orientifold string theories on circles and $2$-tori. In these
theories, D-brane charges lie in twisted $KR$-groups of $(X,\iota)$,
where $X$ is the spacetime manifold and $\iota$ is the involution on
$X$ defining the orientifold structure. (That D-brane charges for
orientifolds are classified by $KR$-theory was pointed out in 
\cite[\S5.2]{Witten:1998}, \cite{Hori:1999me}, and \cite{Gukov:1999},
but twisting (as defined in
\cite{MoutuouThesis,2011arXiv1110.6836M,2012arXiv1202.2057M} and
\cite{Doran:2013sxa}) may arise due to the $B$-field, 
as in \cite{Witten:1998-02}, and/or the charges of the
$O$-planes, as explained in \cite{Doran:2013sxa}.) 
These orientifold theories were found in \cite{MR3316647} to
split up into a number of T-duality groupings, with the theories in
each grouping all related to one another by various T-dualities. The
twisted $KR$-groups attached to each of the theories within a
T-duality grouping were all found to be isomorphic to one another, up
to a degree shift.  Actual isomorphisms of the twisted $KR$-groups 
for the theories within a T-duality group were constructed in
\cite{MR3305978}, 
where we showed that all of these $K$-theory isomorphisms arise from
cases of the real Baum-Connes conjecture for solvable (in fact,
also virtually abelian) discrete groups.  (By the way, these are cases
where the real Baum-Connes conjecture was already known to hold, by
\cite{MR2082090,MR2077669}.)  This sort of duality
involving $KR$-theory already
appeared in \cite[Theorem 2.5]{MR842428} in the case where
the group involved is free abelian. In that theorem, real Baum-Connes
was shown to give an isomorphism from $KO$-homology of a torus to
$KR$-cohomology of another torus of the same dimension (really the
T-dual torus).  This isomorphism turns out to be related to T-duality
between the type I and type IA string theories on the circle.

The work of Karoubi and Weibel \cite{MR1958527} showed that
$KR$-theory of complex algebraic varieties (equipped with an
antiholomorphic involution) is closely related to the algebraic
$K$-theory of the associated \emph{real} algebraic variety.  This
raised the tantalizing possibility that T-duality of orientifolds is
closely related to isomorphisms of algebraic $K$-groups of real
varieties or to equivalences of derived categories of (twisted)
coherent sheaves on real varieties.  The purpose of the paper is to
verify this conjecture in the case of real affine rational curves
(modeling orientifolds on the circle) and real elliptic curves.

After the preliminary Section \ref{sec:KR}, we begin in Section
\ref{sec:circles} with the case of certain rational affine curves over
$\bR$, which give algebraic models for the four orientifolds on the
circle. We compute the algebraic $K$-theory for all of these
``algebraic circles.'' While
the $K$-theory spectra for dual algebraic circles are not exactly the
same, they do become equivalent after 
introducing finite coefficients. Then in Section \ref{sec:ellcurve},
we consider an algebraic analogue of the duality between the type IIA
orientifold theory on a Klein bottle and the $\tIA$ theory, which
lives on a split elliptic curve over $\bR$ with a nontrivial twist
given by a certain Azumaya algebra, described in Proposition
\ref{prop:Azellcurve}. The main result, Theorem \ref{thm:main},
is an equivalence of derived categories of twisted coherent sheaves,
which is considerably stronger (see for example \cite[Example
  V.3.10.2]{MR3076731}) than a mere isomorphism of algebraic
$K$-groups.  This result can be interpreted as saying that the
``dual abelian variety'' to a smooth projective genus-$1$ curve over
$\bR$ with no real points is a noncommutative Azumaya algebra over a
split elliptic curve over $\bR$.  This result is in many ways
reminiscent of the result of C{\u{a}}ld{\u{a}}raru \cite[Chapter
  4]{MR2700538}, \cite{MR1887894} that in the case of an elliptic
fibration without a 
section, a certain Azumaya algebra over the relative Jacobian
is dual to the original elliptic fibration, and there is an
associated equivalence of derived categories of [twisted] coherent 
sheaves. (The genus-$1$
curve with no real points is like an elliptic fibration without a
section, and the split elliptic curve is its Jacobian, as we will see
in Section \ref{sec:ellcurve}.)  The result is also closely related to
the results of \cite{2014arXiv1409.2580A}, though it doesn't quite
seem to be covered by the hypotheses imposed there.

\subsection{Notational convention}
\label{sec:notation}
Since this paper uses both topological and algebraic $K$-theory, we
need a way to distinguish them.  $K$-groups with the index \emph{up},
such as $K^{-\bullet}$, $KO^{-\bullet}$, $KR^{-\bullet}$, and
$KSC^{-\bullet}$, always denote \emph{topological} $K$-groups. $K$-groups with
the index \emph{down}, such as $K_\bullet(R)$, always denote
\emph{algebraic} $K$-theory of rings or schemes.  (All the schemes
discussed here will be smooth, so it will not be necessary to
distinguish this from $G$-theory or $K'$-theory.)
When we need topological $K$-homology (this only shows up once in the
paper), we denote it by $K\top_\bullet$.
The notation $\bK$ denotes the algebraic $K$-theory
spectrum (of a ring, scheme, etc.), and $\bK\top$ denotes the
(non-connective) topological $K$-theory spectrum.

\subsection{Review of KR-theory}
\label{sec:KR}
We quickly remind the reader about Atiyah's $KR$-theory
\cite{MR0206940}. This is the topological $K$-theory (with compact supports)
of Real vector bundles $E$ over Real spaces $(X, \iota)$. A Real space
is just a locally compact Hausdorff space $X$ equipped with a
self-homeomorphism $\iota$ satisfying $\iota^2=\id_X$. A Real vector
bundle $E$ over such a space is a complex vector bundle equipped with
a conjugate-linear vector bundle automorphism of period $2$
compatible with $\iota$. The
$KR$-theory of $(X, \iota)$ can be identified with the topological
$K$-theory of the real Banach algebra $C_0(X,\iota) =\{f\in C_0(X)
\mid f(\iota(x)) = \overline{f(x)}\}$. We shall use 
the indexing convention of \cite{MR1031992,Doran:2013sxa,MR3316647}:
$\bR^{p,q}$ denotes $\bR^p\oplus 
\bR^q$ with the involution that is $+1$ on the first summand and $-1$
on the second summand, and $S^{p,q}$ denotes the unit sphere in
$\bR^{p,q}$.  Thus $S^{0,2}$, for example, denotes the circle $S^1$
with its antipodal involution.  The two indices $p$ and $q$ play
complementary roles, in that $KR^\bullet(X\times \bR^{p,q}) \cong
KR^{\bullet-p+q}(X)$ \cite[Theorem 2.3]{MR0206940}.
Atiyah \cite[Proposition 3.5]{MR0206940} also
showed that for any locally
compact space $X$, $KR^{-\bullet}(X\times S^{0,2})\cong
KSC^{-\bullet}(X)$, the self-conjugate $K$-theory of Anderson \cite{Anderson}
and Green \cite{MR0164347}.

Karoubi and Weibel \cite[Main Theorem 4.8]{MR1958527} showed that
$KR$-theory is
intimately connected with \emph{algebraic} $K$-theory of algebraic
varieties defined over $\bR$. In fact, if $X$ is a smooth projective
variety defined over $\bR$, then $K_\bullet(X, \bZ/m)\cong 
KR^{-\bullet}(X(\bC), \bZ/m)$ in the stable range $\bullet\ge \dim X$.
Here the set $X(\bC)$ of complex points is given its (compact Hausdorff)
analytic topology and the natural action of $G=\Gal(\bC/\bR)$.
They proved the theorem for $m$ a power of two, which is the most
interesting case, though it also follows for $m$ an odd prime by
Bloch-Kato. The theorem holds for affine varieties as well, provided
one modifies the definition of $KR^{-\bullet}(X(\bC))$ a bit.  Since
$X(\bC)$ is noncompact in this case, one needs to use
homotopy-theoretic $KR$-theory $KR_h$ instead of $KR$-theory with compact
supports. ($KR_h^{-\bullet}(X)\cong KO^{-\bullet}(\pt)$ if $X$ is
equivariantly, but not necessarily properly, contractible, so in this
theory, all affine spaces $\bA^n(\bR)$ behave like $\Spec \bR$.)

\section{Algebraic circle orientifolds}
\label{sec:circles}

\subsection{An algebraic analogue of circle orientifolds}
\label{sec:circle}

We begin with the case of orientifolds on a circle, with the
involution coming from a linear involution on $\bR^2$, restricted to
the unit circle. It was found in \cite{Gao:2010ava,Doran:2013sxa,MR3316647}
that there are four such orientifold theories, known in the physics
literature as types I, $\tI$, IA, and
$\tIA$. These split into two T-duality groupings, one of
which contains theories I and IA, corresponding to the Real spaces
$S^{2,0}$ and $S^{1,1}$, and the other of which contains theories
$\tI$ and $\tIA$, corresponding to the Real spaces
$S^{0,2}$ and $S^{1,1}_{(+,-)}$. Here the subscript $(+,-)$ in
$S^{1,1}_{(+,-)}$ indicates 
that of the two $O$-planes in $S^{1,1}$ (i.e., fixed points for the
involution), one has been given a plus sign (meaning that the
Chan-Paton bundle there is of real type) and one has been given a
minus sign (meaning that the Chan-Paton bundle there is of quaternionic type).

What are the algebraic analogues of these Real spaces?

Let $W$, $X$, and $Y$ be the following smooth affine curves defined over
$\bR$:
\[
\begin{aligned}
W&=\Spec \left(\bR[x,y]/(x^2+y^2-1)\right),\\
X&=\Spec \left(\bR[x,y]/(x^2+y^2+1)\right),\\
Y&=\Spec \left(\bR[v, v^{-1}]\right)\cong \bG_m.
\end{aligned}
\]
After applying base change $\square \rightsquigarrow
\Spec \bC \times_{\Spec \bR} \square \,$, 
$W$, $X$, and $Y$ all become 
$\Spec \left(\bC[v, v^{-1}]\right)\cong \bG_m$. For $Y$ this is obvious.
For $X_\bC$ we use the isomorphism sending $(x,y)\in \bA^2(\bC)$ with
$x^2+y^2+1=0$ to $x+iy\in  \bG_m(\bC)$; note that $z=x+iy$ is
invertible and determines the pair $(x,y)$ since $(x+iy)(iy-x) = 1$,
hence $x = (z - z^{-1})/2$ and $y = (z + z^{-1})/(2i)$.  Similarly,
for $W_\bC$ we use the isomorphism sending $(x,y)\in \bA^2(\bC)$ with
$x^2+y^2-1=0$ to $x+iy\in  \bG_m(\bC)$; note that $z=x+iy$ is
invertible and determines the pair $(x,y)$ since $(x+iy)(x-iy) = 1$,
hence $x = (z + z^{-1})/2$ and $y = (z - z^{-1})/(2i)$.

In accordance with the point of view of Karoubi and Weibel,
we view $W(\bC)$, $X(\bC)$ and $Y(\bC)$ as Real spaces in the sense of
Atiyah. Of course in all cases we just have $\bC^\times$
as the underlying space, but the
Galois action of $G=\Gal(\bC/\bR)$ is different.  On $Y(\bC)$, the
action of $G$ is usual complex conjugation on $\bC^\times$, and the
fixed set (the union of the ``O-planes'') is $\bR^\times = (-\infty, 0)
\sqcup (0,\infty)$.  Thus $Y$ is an algebraic model for the Real space
$S^{1,1}$. On $X(\bC)$, on the other hand, the Galois action must be
free, since $X(\bR)=\emptyset$.  (There are no points $(x,y)\in
\bA^2(\bR)$ with $x^2+y^2+1=0$.)  In fact one can see that when we
identify $X(\bC)$ with $\bC^\times$, $G$ acts by $z\mapsto -\bar z^{-1}$.
$X(\bC)$ has a (non-proper) $G$-equivariant 
deformation retraction down to a copy of $S^1$, identified with
\[
S^1 = \{(ix,iy)\mid  (x,y)\in \bA^2(\bR),\ x^2+y^2 = 1\},
\]
and on this copy of $S^1$, $G$ acts by $(ix,iy)\mapsto (-ix, -iy)$,
i.e., by the antipodal map, so $X$ is an algebraic model for
$S^{0,2}$. The case of $W(\bC)$ is similar, but in this case the Galois
action (if we identify $W(\bC)$ with $\bC^\times$) 
is given by $z\mapsto \bar z^{-1}$, and $G$ acts trivially on
\[
S^1 = \{(x,y)\mid  (x,y)\in \bA^2(\bR),\ x^2+y^2 = 1\}.
\]
As shown in \cite[Theorem A.1]{MR1958527}, $W(\bC)$ has a
$G$-equivariant deformation retraction down to this copy of
$S^{2,0}$, so $W$ is an algebraic model for $S^{2,0}$. 

By \cite{MR1958527}, the algebraic $K$-theories of $W$, $X$,
and $Y$ \emph{with finite coefficients} in positive degrees
thus coincide with $KR^{-\bullet}(S^{2,0})\cong KO^{-\bullet}\oplus KO^{-\bullet-1}$,
$KR^{-\bullet}(S^{0,2})\cong KSC^{-\bullet}$ and with
$KR^{-\bullet}(S^{1,1})\cong KO^{-\bullet}\oplus KO^{-\bullet+1}$,
respectively.

We also want an algebraic model for $S^{1,1}_{+,-}$.  This requires a
noncommutative twisting of $Y$, given by an Azumaya algebra.
Now it was shown in \cite{MR0412193} that $\Br Y\cong (\bZ/2)^2$,
since $Y(\bR)$ has two connected components
(in the Hausdorff topology). In other words, there is
an Azumaya algebra over $Y$ whose splitting over the two connected
components of $Y(\bR)$ is whatever one would like.  To get the sign
choice $(+,-)$ over the two O-planes, we require an Azumaya algebra
split over one component and non-split over the other.
In fact, we can write this down very explicitly; let $\cA=\bR\langle t, t^{-1},
  u\rangle/(u^2 = -1, ut = -tu)$.  This is a noncommutative
Noetherian $\bR$-algebra with center $Z(\cA)=\bR[t^2, t^{-2}]\cong
\bR[v, v^{-1}]$, in fact
a quaternion algebra over $Z(\cA)$ in the sense of \cite{MR0429992},
\cite{MR0447339}, or \cite{MR570179}.  Over a point in $Y(\bR)$,
corresponding to setting $t^2\mapsto \alpha\in \bR^\times$, splitting
of $\cA$ is determined by the Hilbert symbol $(\alpha,\, -1)
= \begin{cases} +1, & \alpha>0\\ -1, & \alpha<0 \end{cases}$,
and thus the algebra $\cA$ does indeed correspond to the sign choice
$(+,-)$. It follows that with finite coefficients, the algebraic
$K$-theory of $\cA$ in positive degrees agrees with
$KR^{-\bullet}_{(+,-)} (S^{1,1})\cong KSC^{1-\bullet}$.

\subsection{The first duality}
\label{sec:circleduality1}
The first duality for orientifold theories on the circle relates the
type I and IA theories, and is reflected in the $KR$-theory
isomorphism $KR^{1-\bullet}(S^{2,0})=KO^{1-\bullet}(S^1)
\cong KR^{-\bullet}(S^{1,1})$.  We showed in \cite{MR3305978} that
this comes from the Baum-Connes isomorphism $KO\top_\bullet(B\bZ)
\xrightarrow{\cong} KO\top_\bullet(C^*_\bR(\bZ))\cong KR^{-\bullet}(S^{1,1})$.
Via the Karoubi-Weibel connection between $KR$ and algebraic
$K$-theory of real varieties, this suggests comparing
$K_\bullet(Y)$ and $K_{\bullet-1}(W)$. By the Fundamental
Theorem of Algebraic $K$-Theory
(recall that all of $W$, $X$, and $Y$ are smooth), we have
\[
K_\bullet(Y)=K_\bullet\left(\bR[v, v^{-1}]\right)\cong
K_\bullet(\bR) \oplus K_{\bullet-1}(\bR).
\]
At the same time, we have
\[
K_\bullet(W)=K_\bullet\left(\bR[x,y]/(x^2+y^2-1)\right)=
K_\bullet(\bR(q)),
\]
where $q(x,y)=x^2+y^2$, in the notation of \cite{MR799254}, which
computes the $K$-theory of quadric hypersurfaces. (In Swan's notation,
$\bR(q)$ is the quotient of $\bR[x,y]$ by the ideal generated by $q(x,y)-1$.)

We certainly cannot have $K_n(Y)\cong K_{n-1}(W)$ when $n=0$,
since $K_0(Y)=\bZ$ and $K_{-1}(W)=0$.  It's also known that
$K_0(W)\cong \bZ\oplus \bZ/2$ (there are many proofs, 
such as \cite[Corollary 10.8 or \S10, Example 3]{MR799254}, but for a
completely elementary argument see \cite[Exercise 1.4.23]{MR1282290}),
while we see that $K_1(Y) \cong K_1(\bR) \oplus K_0(\bR)
\cong \bR^\times \oplus \bZ$, which surjects onto $K_0(W)\cong
\bZ\oplus \bZ/2$ with completely divisible kernel.

For a complete calculation of $K_\bullet(W)$ we can use
\cite{MR799254}. Note that $W$ is an open affine subscheme of the
(Severi-Brauer)
projective variety over $\bR$ defined by the homogeneous equation
$x^2+y^2 - z^2 = 0$ in $\bP^2$.  By \cite[Example V.1.6.1]{MR3076731}, this
variety is simply $\bP^1(\bR)$.  The complement $\bP^1\smallsetminus
W$ has no real points and two conjugate complex points; as a scheme it
is just $\Spec \bC$.  By Quillen's calculation of $K$-theory for
projective spaces (\cite[\S8]{MR0338129} or \cite[Theorem
  V.1.5]{MR3076731}), $K_\bullet(\bP^1(\bR)) 
\cong K_\bullet(\bR) \oplus K_\bullet(\bR) $.  Thus we
derive the exact sequence (which also appears as \cite[Theorem
  2]{MR799254}): 
\begin{equation}
\cdots \to K_n(\bC) \xrightarrow{\alpha} K_n(\bR) \oplus K_n(\bR) \to
K_n(W) \xrightarrow{\partial} K_{n-1}(\bC) \to \cdots. 
\label{eq:KthyW}
\end{equation}
Since $W(\bR)\ne \emptyset$, one copy of $K_n(\bR)$ splits off from
$K_n(W)$ and the complement $\tK_n(W)$ is given by the homotopy groups
of the cofiber $\widetilde{\bK}(W)$
of a map $r\co\bK(\bC)\to \bK(\bR)$.  From the
description in \cite[Theorem 2]{MR799254} of the map $\alpha$ in
\eqref{eq:KthyW} as a forgetful map, or from Quillen's description
of the $K$-theory of projective space, one can see that $r$ is just
the realification or transfer map obtained from viewing
finite-dimensional complex vector spaces as real vector spaces of
twice the dimension.

Note by the way that the calculation of $K_\bullet(Y)$ could also be
done by the same method as for $K_\bullet(W)$, as $Y$ is an open
subscheme of the projective nonsingular curve defined by the
homogeneous equation $xy-z^2=0$, which again is $\bP^1(\bR)$.  The
complement $\bP^1\smallsetminus Y$ is $\Spec \bR\sqcup \Spec \bR$, so
the analogue of \eqref{eq:KthyW} is
\[
\cdots \to 
K_n(\bR)\oplus K_n(\bR) \xrightarrow{\alpha} K_n(\bR)\oplus K_n(\bR)
\to K_n(Y) \xrightarrow{\partial} K_{n-1}(\bR)\oplus K_{n-1}(\bR) \cdots
\]
which after removing $ K_n(\bR) \xrightarrow{\cong} K_n(\bR) $ from
$\alpha$ splits into split short exact sequences
\[
0\to K_n(\bR)\to K_n(Y) \rightleftarrows K_{n-1}(\bR) \to 0.
\]

The algebraic $K$-theory version of the T-duality between the type I
and type IA theories on the circle now takes the following form.
\begin{theorem}
\label{thm:IandIAduality}
Let $W = \Spec \left(\bR[x,y]/(x^2+y^2-1)\right)$ be the algebraic
circle and let $Y =\bG_m$ {\lp}over $\bR${\rp}.  The $K$-theory spectrum of
$Y$ splits as $\bK(Y)\cong \bK(\bR)\oplus \Sigma\bK(\bR)$.  The
$K$-theory spectrum of $W$ splits as $\bK(W)\cong \bK(\bR)\oplus
\widetilde{\bK}(W)$, where $\widetilde{\bK}(W)$ is the cofiber of the
realification map $r\co \bK(\bC)\to \bK(\bR)$.  There is a map
of $K$-theory spectra with finite coefficients
$\Omega\bK(\bR;\bZ/m)\to \widetilde{\bK}(W;\bZ/m)$ for which the induced map
on  homotopy groups
\[
K_{n+1}(\bR;\bZ/m)\to \widetilde{K}_{n}(W;\bZ/m)
\]        
is an isomorphism for $n\ge0$ and any $m>0$.

In lowest dimensions we have $K_0(Y)\cong \bZ$, $K_1(Y)\cong
\bR^\times \oplus \bZ$, $K_2(Y)\cong K_2(\bR)\oplus \bR^\times \cong
\bR^\times \oplus \bZ/2 \oplus (\text{uniquely divisible})$,
$K_0(W)\cong \bZ\oplus \bZ/2$,
$K_1(W)\cong \bR^\times \oplus \bZ/2$.
\end{theorem}
\begin{proof}
Most of this is contained in the above discussion, but it remains to
relate the spaces $\Omega\bK(\bR)$ and $\widetilde{\bK}(W)$.
Recall that by \cite{MR772065}, the natural maps
$\bK(\bR)\to \bK\top(\bR)$ and $\bK(\bC)\to \bK\top(\bC)$ are
equivalences in nonnegative degrees with finite coefficients.  Now by
real Bott periodicity, there is a cofiber sequence of spectra
\begin{equation}
\bK\top(\bR)\xrightarrow{c} \bK\top(\bC)\xrightarrow{r\circ
  \beta^{-1}} \Sigma^2\bK\top(\bR),
\label{eq:rcseq}
\end{equation}
where $c$ is ``complexification'' and $\beta\co
\bK\top(\bC)\xrightarrow{\cong} \Omega^2\bK\top(\bC)$ is the Bott map,
associated to a familiar complexification/realification sequence in
topological $K$-theory (see for example \cite[Theorem
  1.18(5)]{MR1935138} or \cite[Theorem III.5.18]{MR0488029}). That
means that the 
cofiber of $r\co \bK\top(\bC)\to \bK\top(\bR)$ can be identified with
$\Omega\bK\top(\bR)$. Because of Suslin's Theorem, the result follows.
\end{proof}

\begin{remark}
\label{rem:nodereqaffine}
One might wonder if there is some equivalence of derived categories
underlying the $K$-theory isomorphism in Theorem
\ref{thm:IandIAduality}, but because of \cite[Proposition
  9.2]{MR1002456}, there are no nontrivial derived equivalences
between affine schemes.  The projective case is different, as we will
see later.
\end{remark}

\subsection{The second duality}
\label{sec:ItildeandIAtilde}
Now we want to look for an algebraic analogue of the duality between
the $\tI$ and $\tIA$ theories on $S^1$.  This means we want to relate
the algebraic $K$-theories of $X$ and the Azumaya algebra $\cA$ over
$Y$.

To begin, we can compute the $K$-theory of $X=\bR(q)$ 
with $q(x,y)=-x^2-y^2$ the same way we did for
$W$.  The major difference is that $X$ has no real points, nor does
its projective completion, the quadric in $\bP^2$ defined by
$x^2+y^2+z^2=0$. Thus this quadric is a Brauer-Severi variety
associated to a non-split Azumaya algebra, namely $\bH$.  So by
Quillen's calculation in \cite[\S8]{MR0338129}, the $K$-theory
spectrum of this variety is $\bK(\bR)\oplus \bK(\bH)$, and as before
in \eqref{eq:KthyW} we get an exact sequence (which again could be
read off from \cite[Theorem 2]{MR799254})
\begin{equation}
\cdots \to K_n(\bC) \xrightarrow{\alpha} K_n(\bR) \oplus K_n(\bH) \to
K_n(X) \xrightarrow{\partial} K_{n-1}(\bC) \to \cdots. 
\label{eq:KthyX}
\end{equation}
Following \cite{MR799254} again, the map $\alpha$ comes from
restriction of modules over Clifford algebras, from $M_2(\bC)$ to the
subalgebras $\bH$ and $M_2(\bR)$.  The first component of $\alpha$ is
$r\co \bK(\bC)\to \bK(\bR)$; the second component is the map
$\bK(\bC)\to \bK(\bH)$ induced by restriction of modules over $M_2(\bC)$
to $\bH$, whose analogue in topological $K$-theory fits into a shifted
version of \eqref{eq:rcseq}.  The lowest-degree portion of
\eqref{eq:KthyX} looks like
\[
\cdots \to \bC^\times \xrightarrow{\alpha} \bR^\times \times
\left(\bH^\times/SU(2)\right) \to K_1(X) \to \bZ \xrightarrow{a\mapsto (2a, a)}
\bZ^2 \to K_0(X) \to 0.
\]
Thus $K_0(X)\cong \bZ$ and $K_1(X) \cong \bR^\times$.  Note that
modulo completely divisible abelian groups these agree with
$KR^{-j}(X(\bC))\cong KSC^{-j}$ for $j=0,1$, as we might expect from
\cite{MR1958527}. In fact, with finite coefficients and in positive
degrees, we can replace algebraic $K$-theory by topological $K$-theory
in \eqref{eq:KthyX} using \cite{MR772065}, and then \eqref{eq:KthyX}
becomes the known exact sequence \cite[III.7.15]{MR0488029}
\[
\cdots \to KSC^{-n-1} \to K^{-n} \to KO^{-n}\oplus KSp^{-n} \to
KSC^{-n} \to \cdots.
\]

We turn now to the $K$-theory of the noncommutative algebra $\cA$ from
Section \ref{sec:circle}. In the notation of \cite{MR799254}, this
algebra is a Clifford algebra $C(q)$ over $R=\bR[t^2, t^{-2}]$
attached to a rank-two quadratic $R$-module $(M,q)$, where if $e_1,
e_2$ are an $R$-basis for the free $R$-module $M$, $q(e_1) = -1$,
$q(e_2)=t^2$, and $e_1\perp e_2$.  Since $t^2$ is a unit in $R$, $q$
is non-singular. Furthermore, $C(q)\cong C_0(q\oplus\langle
-1\rangle)$ by \cite[Lemma 4.5]{MR799254}.  Thus we can apply
\cite[Theorem 9.1]{MR799254} with $d=1$ to get
$\bK(\cA) \oplus \bK(R)\cong \bK(Z)$. Here 
$Z = X(q\oplus\langle-1\rangle)$ is the closed subscheme of $\bP^2(R)$
defined by the homogeneous equation $q\oplus\langle-1\rangle=0$.
Alternatively, by \cite[\S8, Theorem 4.1]{MR0338129},
$\bK(\cA) \oplus \bK(R)\cong \bK(Z)$, where $Z$ the Brauer-Severi scheme
over $Y$ associated to $\cA$, and this turns out to be the same as
above. So we're basically reduced to computing the $K$-theory of $Z$.

This is still a formidable task, so it is easiest to take another
approach and to regard $\cA$ as a crossed product or
twisted Laurent polynomial extension $\bC \rtimes_\sigma \bZ =
\bC_\sigma [t, t^{-1}]$, where $\bC = \bR[u]/(u^2+1)$ and $\sigma$ is
complex conjugation (since conjugation by $t$ sends $u$ to $-u$).
Thus by the ``algebraic Pimsner-Voiculescu sequence'' of
\cite[Theorems 2.1 and 2.3]{MR929766} or \cite[Corollary
  2.2]{MR1325783} (Nil terms drop out since everything is regular),
$\bK(\cA)$ is equivalent to the cofiber of $\bK(\bC) 
\xrightarrow{1-\sigma_*} \bK(\bC)$.  In particular, in low degrees
we get
\[
\cdots \to \star \to \star \to K_2(\cA) \to 
\bC^\times \xrightarrow{z\mapsto z/\overline z} \bC^\times \to K_1(\cA)
\to \bZ \xrightarrow{0} \bZ \to K_0(\cA) \to 0,
\]
where $\star=K_2(\bC)$ is uniquely divisible,
so $K_2(\cA)\cong \bZ/2\oplus (\text{uniquely divisible})$, 
$K_1(\cA)\cong \bR^\times_+\oplus \bZ$ and $K_0(\cA)\cong \bZ$.

The algebraic $K$-theory version of the T-duality between the 
$\tI$ and $\tIA$ theories on the circle now takes the following form.
\begin{theorem}
\label{thm:ItandIAtduality}
Let $X = \Spec \left(\bR[x,y]/(x^2+y^2+1)\right)$, a smooth affine
quadric curve with no real points, and let $\cA$ be the quaternion
algebra over $Y =\bG_m$ given by $\cA=\bR\langle t, t^{-1},
u\rangle/(u^2 = -1, ut = -tu)$. The $K$-theory spectrum of $\cA$ is
the cofiber of $\bK(\bC)\xrightarrow{1-\sigma_*} \bK(\bC)$, where
$\sigma $ is complex conjugation.  The $K$-theory spectrum of $X$
is the cofiber of $\bK(\bC)\xrightarrow{\alpha} \bK(\bR)\oplus
\bK(\bH)$, where $\alpha$ comes from restriction of Clifford modules,
as explained above.  There is a map
of $K$-theory spectra with finite coefficients
$\Omega\bK(\cA;\bZ/m)\to {\bK}(X;\bZ/m)$ for which the induced map on
homotopy groups
\[
K_{n+1}(\cA;\bZ/m)\to {K}_{n}(X;\bZ/m)
\]        
is an isomorphism for $n\ge0$ and any $m>0$.

In lowest dimensions we have $K_0(X)\cong \bZ$, $K_1(X)\cong
\bR^\times$, $K_0(\cA)\cong \bZ$, $K_1(\cA)\cong \bR^\times_+ \oplus \bZ$.
\end{theorem}
\begin{proof}
Again, most of this is in the discussion above. After applying
\cite{MR772065} and going to the stable range (so we can ignore the
difference between connective and periodic $K$-theory spectra), we see
that with finite coefficients, $\bK(\cA)$ becomes the cofiber of
$\bK\top(\bC)\xrightarrow{1-\sigma_*} \bK\top(\bC)$, which is
$\Sigma\bK\bS\bC$, $\bK\bS\bC$ the $4$-periodic spectrum of
self-conjugate $K$-theory.  Similarly $\bK(X)$ becomes
the cofiber of $\bK\top(\bC)\xrightarrow{r} \bK\top(\bR)\oplus
\bK\top(\bH)$, which is $\bK\bS\bC$, and the result follows.
\end{proof}

\section{Some elliptic curve orientifolds}
\label{sec:ellcurve}

Now we consider cases of duality related to elliptic curves.
Generally speaking, a duality between a type IIA theory and a type IIB
theory on elliptic curves should in the world of algebraic mirror
symmetry be reflected in an equivalence between a category of coherent
sheaves on one side and a Fukaya category on the other 
\cite{MR1633036}. Since the
analogue of Fukaya categories in the orientifold case is not well
developed yet, we leave such dualities as a project for future work.
But there is one interesting case of an orientifold duality that we
should be able to treat without going out of categories of coherent
sheaves. This involves the T-duality grouping
\cite[\S4.1.2]{MR3316647} containing the type IIA theory on an
elliptic curve with free antiholomorphic involution and the
type IIA theory on an elliptic curve with antiholomorphic involution 
with exactly two fixed circles, one given a $+$ charge and one given a
$-$ charge (the $\tIA$ theory). Since an elliptic curve with
antiholomorphic involution 
is the same thing as a smooth projective genus-$1$ curve defined over
$\bR$, there are algebraic models for these theories involving
coherent sheaves on real varieties, though in one case we need to
twist by a suitable Azumaya algebra.  That such an Azumaya algebra
exists is guaranteed by an application of \cite[Theorem
  3.6]{MR1869390}, as  indicated in the following Proposition.
Note that in what follows, $\cA$ represents a different Azumaya algebra
than in Section \ref{sec:circles} above.  We recall that a smooth
projective curve defined over $\bR$ has an invariant called the
\emph{species}, which is the number of connected components of the
set of real points (in the analytic topology). This invariant can take
any value from $0$ up to $g+1$ ($g$ the genus), by a classical
theorem of Harnack (see also \cite[Proposition 3.1(1)]{MR631748}).
\begin{proposition}
\label{prop:Azellcurve}
Let $E$ be a split elliptic curve defined over $\bR$ {\lp}i.e., in the
language of \cite{MR640091}, a real elliptic curve of species $2$, so
that $E(\bR)$ is topologically a disjoint union of two
circles{\rp}. Then there is a unique {\lp}up to Morita equivalence{\rp}
Azumaya algebra $\cA$ over $E$
representing an element of $\Br(E)$ of order $2$, such that $\cA$ is
split over the component $E(\bR)^+$ of $E(\bR)$ containing the identity element
$e$ {\lp}for the group structure on $E(\bC)${\rp} and non-split over
the other component $E(\bR)^-$ of $E(\bR)$.
\end{proposition}
\begin{proof}
A split elliptic curve defined over $\bR$ has Weierstra{\ss} equation
\[y^2 = (x-\alpha)(x-\beta)(x-\gamma),\] 
with $\alpha<\beta<\gamma$ distinct points in $\bR$.
The $2$-torsion subgroup $M={}_2E(\bC)$ of $E(\bC)$ consists of $e$
(the point at $\infty$) and the points $a=(\alpha,0)$, $b=(\beta,0)$, and
$c=(\gamma,0)$, with $c$ 
and $e$ in one component of $E(\bR)$, which we'll call $E(\bR)^+$, and
$a$ and $b$ in the other component $E(\bR)^-$. Let
$G=\Gal(\bC/\bR)$. By \cite[Proposition 3.6]{MR1219265}, $\Br(E)\cong
(\bZ/2)^2$; in fact, as Pedrini and Weibel point out, this was really
computed by Witt in \cite[II$'$ and III$'$]{MR1581415}.  More precisely,
via an analysis of the $G$-action on the exact sequence
\[
0 \to M \to E(\bC) \xrightarrow{2} E(\bC) \to 0,
\]
it is shown in \cite{MR1869390} that ${}_2\!\Br E ={}_2\!\Br \bR \oplus{}
_2H^1(G, E(\bC))$.  As a $G$-module, $E(\bC)$ looks like
$(\bT,\id)\times (\bT, \sigma)$, where $\bT$ as usual is the circle
group, $\id$ is the trivial $G$-action, and $\sigma$ is complex
conjugation. So 
\[
H^1(G, E(\bC)) = \frac{\ker \bigl( (z,w) \mapsto (z,w)(z,w^{-1})
  \bigr)}{\operatorname{image}\bigl( (z,w) \mapsto (z,w)(z^{-1},w)
  \bigr)} \cong \bZ/2,
\]
and $_2\!\Br E$ is exactly the group of ``sign choices'' discussed in
\cite{Doran:2013sxa} and \cite{2012arXiv1202.2057M}.  Thus there is a
unique Brauer class of Azumaya algebras $\cA$ split over $E(\bR)^+$ and
nonsplit over $E(\bR)^-$, which can be chosen to be represented by the
quaternion algebra over the function field $\bR(E)$ given by the
quaternion algebra symbol $(-1, x-\gamma)$.  (Note that $x-\gamma$ is positive
on $E(\bR)^+$, so the Hilbert symbol there is trivial, while it is
negative on $E(\bR)^-$, so the Hilbert symbol there is nontrivial.)
Viewed as a sheaf of $\cO_E$-modules, $\cA$ is locally free of rank
$4$, and its complexification is isomorphic to $M_2\left(\cO_{E_\bC}\right)$.
\end{proof}

Now we are ready for what is really the main theorem of this paper. It
says that in some sense, the ``dual abelian variety'' to a genus-$1$
curve over $\bR$ with no real points is the noncommutative Azumaya
algebra of Proposition \ref{prop:Azellcurve}.
\begin{theorem}
\label{thm:main}
Let $C$ be a smooth projective genus-$1$ curve over $\bR$ with no real
points {\lp}i.e., in the terminology of \textup{\cite{MR640091}}, a
non-classical real elliptic curve of species $0${\rp}.
Let $E$ be the split elliptic curve over $\bR$ with the same
$j$-invariant. Let $\cA$ be the Azumaya algebra over $E$ as identified
in \textup{Proposition \ref{prop:Azellcurve}}.  Let $\cD(C)$ be the
bounded derived category of {\lp}complexes of{\rp} coherent sheaves on
$C$, and let $\cD(E,\alpha)$ be the bounded derived category of
{\lp}complexes of{\rp} $\alpha$-twisted coherent sheaves on $E$,
where $\alpha$ is the Brauer group class of $\cA$. Then
the categories $\cD(C)$ and $\cD(E,\alpha)$ are equivalent
{\lp}in a way we will make explicit{\rp}, and in
particular, $K_\bullet(C)\cong K_\bullet(\cA)$.
\end{theorem}
Since there are several steps to the proof, we separate out the first part
as a lemma.
\begin{lemma}
\label{lem:Jacobian}
Let $C$ and $E$ be as in \textup{Theorem \ref{thm:main}}.
Then the Jacobian variety of $C$ can be identified with $E$,
and $\Pic^0(C)$ {\lp}the group of 
classes of $\Gal(\bC/\bR)$-fixed divisors on $C$ of degree
zero{\rp} can be identified with the connected component of the
identity $E(\bR)^+$ {\lp}which has index $2${\rp} in the group
$E(\bR)$ of real points of $E$.
\end{lemma}
\begin{proof}[Proof of Lemma]
As before we let $G=\Gal(\bC/\bR)$.
We begin by recalling a number of facts from \cite{MR631748} and
\cite{MR1219265} about the Picard group of $C$.  First of all,
$C_\bC = C\times_{\Spec \bR}\Spec \bC$ is a complex elliptic curve and
as such is self-dual. So the Picard group $\Pic(C_\bC)$ of $C_\bC$ can be
identified with $\Pic^0(C_\bC)\times \bZ$, with $\Pic^0(C_\bC)\cong
C_\bC$. As shown in \cite[Proposition 2.2]{MR631748} and
\cite[Proposition 1.1]{MR1219265}, the Picard
group $\Pic(C)$ of $C$ is a subgroup of $\Pic(C_\bC)$, but has two unusual
features.  First of all, the image of $\Pic(C)$ in
$\Pic(C_\bC)/\Pic^0(C_\bC) \cong \bZ$ is of index two, i.e., all
algebraic line bundles over $C$ (defined over $\bR$)
have even degree.  Secondly,
$\left(\Pic^0(C_\bC)\right)^G$ contains the connected component of the
identity $\Pic^0(C)$ in $\Pic(C)$ with index $2$.  The connected 
component of the identity, denoted $\Pic^+(\bR)$ in \cite{MR631748},
consists of classes of $G$-fixed divisors, but there exist divisors
whose class modulo rational equivalence is fixed by $G$ without the divisor
itself being fixed by $G$. The upshot of this discussion is that the
Jacobian variety of $C$ is a real elliptic curve of species $2$. It has
to complexify to $C_\bC$, so it has the same $j$-invariant as $C$, and
can be identified with $E$, with $\Pic^0(C)$ identified with $E(\bR)^+$.
In fact, there is an \'etale morphism
$\pi$ (an isogeny) from $C$ onto $E$, which on the complexification $C_\bC$ is
just multiplication by $2$.  Complex analytically, $C$ can be
identified with $\bC/(\bZ + \tau \bZ)$, where $\tau=i\tau_0$ with
$\tau_0 > 0$, and the action of $G$ on $C$ is given by $\ts\co z\mapsto
\overline z + \frac12$.  (See for example \cite[Chapter 12]{MR640091}
and \cite[\S4.1.2]{MR3316647} for the physical interpretation.) The
map $\pi\co z\mapsto 2z$ intertwines this $G$-action with the $G$-action
$\sigma\co z\mapsto \overline z$ on the same torus, which corresponds to the
species $2$ curve $E$.  In purely algebraic language, following the
notation in the proof of Proposition \ref{prop:Azellcurve}, let
$M=\{e,a,b,c\}$ be the $2$-torsion subgroup of $E_\bC$, where $e$ is
the identity element (corresponding to the point at infinity in the
Weierstra{\ss} form) and $e,c\in E(\bR)^+$, $a,b\in E(\bR)^-$.
Then the Galois involution $\sigma$ on $E_\bC$ fixes $M$ (pointwise),
and the Galois involution $\ts$ on $C_\bC$ is given by $z\mapsto
\sigma(z)+c$. The map from $C_\bC$ to its Jacobian can be identified
with $\phi\co z\mapsto T_z^*\cL\otimes \cL^{-1}$
(see \cite[Chapter II]{zbMATH05358709}), where $T_z$ is translation
by $z$ and $\cL$ is the line bundle of degree $2$ defined by the
divisor $[e]+[c]$. This divisor is $\ts$-invariant, so that $\cL$ is
defined over $\bR$, as is $\phi(z)$ for $z\in E(\bR)$.  Note that
the kernel of $\phi$ can be seen to be $M$, so again we see that the
Jacobian of $C$ is $C/M\cong E$, and $\phi$ can be identified with
$\pi\co C\to E$.
\end{proof}
\begin{proof}[Proof of \textup{Theorem \ref{thm:main}}]
Recall from \cite[Theorem 1.3.5]{MR2700538} that we may identify
$\cA$ with the endomorphism algebra bundle of $\cE$, where $\cE$ is a
rank-$2$ $\alpha$-twisted locally free sheaf over $E$. (We view
$\alpha$ as a class in $H^2_{\text{an}}(E, \cO_E^\times)$, cohomology
for the analytic topology; cf.\ \cite[Theorem 1.1.4]{MR2700538}.) Furthermore,
there is an equivalence of categories between coherent $\cA$-modules
and coherent $\alpha$-twisted sheaves over $E$, which we will use
hereafter without special comment.

With these preliminaries out of the way, we now define a twisted
bundle $\cP$ on $C\times E$, playing the same role as the Poincar\'e 
bundle in Mukai duality \cite{MR607081}, that will implement the derived
equivalence $\cD(C)\xrightarrow{\cong}
\cD(\cA)\cong \cD(E,\alpha)$, following the paradigm in 
\cite{MR1921811,MR1998775} and \cite{MR2700538,MR1887894}. 
Identify $C_\bC$ and $E_\bC$ as usual and let $\cP_0$ be the 
Poincar\'e bundle on $C_\bC \times E_\bC$, which makes sense
since $E_\bC$ is the Jacobian variety of $C_\bC$. In other words, for $x
\in E_\bC$, the restriction of $\cP_0$ to $C_\bC \times \{x\}$ is 
the line bundle $\cV_x$ defined by $x$ over $C_\bC$.  According to the
the recipe in \cite[Chapter II, \S8]{zbMATH05358709}, $\cP_0$
is defined by the divisor
\[
D_0 =[\Delta] - [C_\bC \times \{e\}] - [\{e\}\times E_\bC ],
\]
where $\Delta$ is the \emph{anti-diagonal} $\{(x,-x):x\in C_\bC \}$.
Note that $\Delta$ can also be described as $m^{-1}(e)$, where $m$ is
the additive group law on $C_\bC = E_\bC$. Now $\sigma\times \sigma$
obviously maps $\Delta$ to itself, so the Galois involution 
$\ts\times \sigma$ on $C_\bC \times E_\bC$
sends $\Delta$ to $\Delta+(c,e)$ and \emph{vice
  versa}. Similarly, $\ts\times \sigma$ maps 
$C_\bC \times \{e\}$ to itself,
and interchanges $\{e\}\times E_\bC$ and $\{c\}\times E_\bC$. 
So the divisor $D_0$ is \emph{not} invariant under $G$ and the
conjugate bundle $\overline{\cP_0}$ is associated to the
divisor
\[
(\ts\times \sigma)(D_0) =[\Delta+(c,e)] - [C_\bC \times \{e\}] -
  [\{c\}\times E_\bC ]. 
\]
Let $\cP=\cP_0\oplus \overline{\cP_0}$.  Even though the individual
summands $\cP_0$ and $\overline{\cP_0}$ are not defined over $\bR$,
$\cP$ is defined over $\bR$, with $G$ interchanging the two summands. For $x
\in E_\bC$, the restriction of $\cP$ to $C_\bC \times \{x\}$ is 
$\cV_x\oplus \overline{\cV_x}$, where $\cV_x$ defined by the element
$x$ of the Jacobian of $C_\bC$.  In other words,
$\cV_x=\phi(z)=T_z^*\cL\otimes \cL^{-1}$, where $2z=-x$
and $z$ is uniquely defined modulo $M$. Since $\cP$ is defined over
$\bR$, it is an $\cO_{C\times E}$-module.  (Note that $\cO_{C\times E}=
\cO_C\boxtimes\cO_E$, the external tensor product.)
We just need to explain how to make $\cP$ into an
$\cO_C\boxtimes\cA$-module. Since $\cP = \cP_0\oplus \overline{\cP_0}$, 
multiplication by $i$ on $\cP$ acts by the
matrix $\cI=\begin{pmatrix}i&0\\0&-i\end{pmatrix}$.
Recall that the gerbe $\alpha$
defining $\cA$ is trivial on $E(\bR)^+$ but nontrivial on
$E(\bR)^-$. At the same time, for $x\in E(\bR)^+$, the associated line
bundle $\cV_x$ is defined over $\bR$. (The divisor of $\cV_x$ is 
$\mu = D_0\cap (C_\bC\times \{x\}) = [-x] - [e]$. This divisor is not
$G$-invariant, but it's rationally equivalent to
$[z] + [z+c] - [e] - [c]$, which \emph{is} 
$G$-invariant.  Here we choose $z\in E(\bR)$
with $2z=-x$, possible since  $x\in E(\bR)^+$.)
So we get a vector bundle isomorphism
$\cV_x\to \overline{\cV_x}$ and thus an endomorphism $J$ of
$\cV_x\oplus \overline{\cV_x}$ anticommuting with $\cI$ and with $J^2=1$. Since
$\cI$ and $J$ satisfy the relations for a split quaternion algebra
over $\bR$, $C_\bC \times E(\bR)^+$ is a module for
$(\cO_E\boxtimes \cA)\cong (\cO_E\boxtimes M_2(\cO_C))$ over $C_\bC \times
E(\bR)^+$. For $x\in E(\bR)^-$, the associated line
bundle $\cV_x$ is not defined over $\bR$, but as explained in the
proof of \cite[Proposition 2.2]{MR631748}, there is an
isomorphism $\varphi_x\co \cV_x\to \overline{\cV_x}$ with $\varphi_x^2
= -1$. (In other words,
we get an endomorphism $J$ anticommuting with $\cI$ and with $J^2=-1$.)
This gives $\cV_x\oplus \overline{\cV_x}$ a quaternionic structure. 
In fact (after renormalizing)
we can make this structure algebraic in the variable $x$.
To see this, again change the divisor $\mu$ of $\cV_x$ up to
rational equivalence to the form $\mu' = [z] + [z+c] - [e] - [c]$ with
$2z=-x$.  Since $x\in E(\bR)^-$, $z$ cannot be chosen in $E(\bR)$.
But $\ts(\mu') = [\sigma(z)+c] + [\sigma(z)] - [c] - [e]$, and
$\mu-\ts(\mu) = [z] - [\sigma(z)] + [z+c] - [\sigma(z)+c]$ is a
principal divisor, in fact associated to the meromorphic function
$f\co w\mapsto \sn(w-z)$, where $\sn$ is the standard Jacobi
elliptic function with
parameters $K=\frac{1}{4}$, $K'=\frac{i\tau_0}{2}$
\cite[\S22.4]{NIST:DLMF}.  This is obviously meromorphic in $z$ and
thus in $x$, and as indicated in \cite{MR631748}, multiplication by
this function implements the quaternionic structure. Indeed,
$w\mapsto \overline{\sn(\ts(w)-z)}$ is also meromorphic in $w$,
and taking $w=z+K$ and using the fact that $\bar z - z = -iK'$
(since $2z\in E(\bR)^-$),
\[
\begin{aligned}
  \sn(w-z) &= \sn(K) = 1,\\
  \overline{\sn(\ts(w)-z)} &= \overline{\sn(\bar w + 2K - z)}\\
  &= \overline{\sn(3K -iK')} = - \overline{\sn(K + iK')} =
  \tfrac{-1}{k} < 0.
\end{aligned}
\]
(The negative sign of the product, what is called $\mathbf{N}f$ in
\cite{MR631748}, is what counts here.) Thus $f$ makes
the restriction of $\cP$ to $C_\bC \times E(\bR)^-$ into a module for
$(\cO_E\boxtimes \cA)\cong (\cO_E\boxtimes \cO_C\otimes \bH)$ over $C_\bC \times
E(\bR)^-$. These $\cO_E\boxtimes \cA$-module structures
on $\cP$ extend to open
neighborhoods of $C_\bC \times E(\bR)^+$ and $C_\bC \times E(\bR)^-$
(in the analytic topology) which, together, cover $C_\bC \times E_\bC$.
Patching these together, we obtain an $\alpha$-twisted sheaf
structure on $\cP$ over $C\times E$.  By the usual process, we
obtain a functor 
\begin{equation}
\Phi\co \cD(C) \to \cD(\cA), \quad
\Phi(\cF) = \mathbf{R}(\pi_E)_*\bigl(\cP\overset{\mathbf{L}}{\otimes}
\mathbf{L}(\pi_C)^*\cF\bigr),
\label{eq:Mukaifunc}
\end{equation}
and a similar functor $\Psi\co \cD(\cA) \to \cD(C)$ obtained by
reversing the roles of $E$ and $C$.

We can now prove the theorem by using \cite[Lemma 2.12]{MR1921811},
modified in the obvious way to deal with twisted sheaves.  This
asserts that the property of $\Phi$ being an equivalence is stable
under base change from $\bR$ to $\bC$.  But the complexification
$\cA_\bC$ of $\cA$ is $M_2(\cO_{E_\bC})$, Morita equivalent
to $\cO_{E_\bC}$, so $\cD(\cA_\bC)$ is equivalent to $\cD(E_\bC)$.
Hence $\Phi$ is an equivalence because
the functor $\Phi_\bC\co \cD(E_\bC) \to \cD(C_\bC)$ is the equivalence given
by Mukai duality \cite{MR607081}.
\end{proof}
\begin{remark}
\label{rem:Antieau}
To link Theorem \ref{thm:main} to the results of \cite{MR1887894} and
\cite{2014arXiv1409.2580A}, note that our genus-$1$ curve $C$ can be
viewed as an elliptic fibration over $\Spec\bR$ with Jacobian $E$,
but without a section, and also as
a torsor over the elliptic curve $E$. To see this,
consider the translation action $m$ of $E_\bC$ on $C_\bC$
(say on the right):
$C_\bC\times E_\bC\xrightarrow{m} C_\bC$, $m(y,x) = y+x$.  This action is
equivariant for the Galois action, since 
\[
m\bigl((\ts\times \sigma)(y,x)\bigr) =
m\bigl(\sigma(y)+c,\sigma(x)\bigr)
= \sigma(y) + c + \sigma(x) = \sigma(y+x) + c = \ts\bigl(m(y,x)\bigr).
\]
Thus we can think of $C$ as a principal homogeneous space for $E$.
In \cite{2014arXiv1409.2580A}, derived equivalences of Brauer twists
for  principal homogeneous spaces over an elliptic curve are studied,
but Theorem 5.1 in that paper, for example, only deals with the case
where the Brauer twist comes from the Brauer group of the field,
$\Br\bR$ in our case. In our situation, $E$ is of course a trivial
torsor over itself, while $C$ has a nontrivial class in
$H^1_{\text{\'et}}(\Spec \bR, E)=H^1(G, E(\bC))\cong \bZ/2$.
(Recall the proof of Proposition \ref{prop:Azellcurve}.)
So our twist $\alpha$ comes from the cokernel of the map
$\Br\bR\to\Br E$, which in the language of Ogg-Shafarevich theory
(see \cite[\S6]{MR1887894}) is the Tate-Shafarevich group
$\sha_{\Spec\bR}(E)\cong \bZ/2$.  Thus it seems that Theorem
\ref{thm:main} is covered neither by the hypotheses in
\cite{2014arXiv1409.2580A} nor by the hypotheses of \cite{MR1887894},
which only deals with algebraically closed fields.  Nevertheless,
the proof of \cite[Theorem 5.1]{MR1887894} could be adapted to give
another proof of Theorem \ref{thm:main}, using the \'etale topology on
the base $\Spec \bR$.
\end{remark}

\bibliographystyle{hplain}
\bibliography{categoricalTduality}

\def\cprime{$'$}
\begin{thebibliography}{10}

\bibitem{MR640091}
Norman~L. Alling.
\newblock {\em Real elliptic curves}, volume~54 of {\em North-Holland
  Mathematics Studies}.
\newblock North-Holland Publishing Co., Amsterdam, 1981.
\newblock Notas de Matem{\'a}tica [Mathematical Notes], 81.

\bibitem{Anderson}
D.~W. Anderson.
\newblock The real {$K$}-theory of classifying spaces.
\newblock {\em Proc. Nat. Acad. Sci. U. S. A.}, 51(4):634--636, 1964.

\bibitem{2014arXiv1409.2580A}
B.~{Antieau}, D.~{Krashen}, and M.~{Ward}.
\newblock Derived categories of torsors for abelian schemes.
\newblock arXiv:1409.2580, 2014.

\bibitem{MR0206940}
M.~F. Atiyah.
\newblock {$K$}-theory and reality.
\newblock {\em Quart. J. Math. Oxford Ser. (2)}, 17:367--386, 1966.

\bibitem{MR2082090}
Paul Baum and Max Karoubi.
\newblock On the {B}aum-{C}onnes conjecture in the real case.
\newblock {\em Q. J. Math.}, 55(3):231--235, 2004.

\bibitem{MR1935138}
Jeffrey~L. Boersema.
\newblock Real {$C^*$}-algebras, united {$K$}-theory, and the {K}\"unneth
  formula.
\newblock {\em $K$-Theory}, 26(4):345--402, 2002, arXiv:math/0208068.

\bibitem{MR2700538}
Andrei C{\u{a}}ld{\u{a}}raru.
\newblock {\em Derived categories of twisted sheaves on {C}alabi-{Y}au
  manifolds}.
\newblock ProQuest LLC, Ann Arbor, MI, 2000.
\newblock Thesis (Ph.D.)--Cornell University.

\bibitem{MR1887894}
Andrei C{\u{a}}ld{\u{a}}raru.
\newblock Derived categories of twisted sheaves on elliptic threefolds.
\newblock {\em J. Reine Angew. Math.}, 544:161--179, 2002.

\bibitem{MR1869390}
V.~Chernousov and V.~Guletski{\u\i}.
\newblock 2-torsion of the {B}rauer group of an elliptic curve: generators and
  relations.
\newblock {\em Doc. Math.}, extra volume, 2001:85--120 (electronic), 2001.
\newblock Proceedings of the {C}onference on {Q}uadratic {F}orms and {R}elated
  {T}opics ({B}aton {R}ouge, {LA}, 2001).

\bibitem{MR0412193}
F.~R. De{M}eyer and M.~A. Knus.
\newblock The {B}rauer group of a real curve.
\newblock {\em Proc. Amer. Math. Soc.}, 57(2):227--232, 1976.

\bibitem{NIST:DLMF}
{NIST Digital Library of Mathematical Functions}.
\newblock \texttt{http://dlmf.nist.gov/}, Release 1.0.10 of 2015-08-07.
\newblock Online companion to \cite{Olver:2010:NHMF}.

\bibitem{Doran:2013sxa}
Charles Doran, Stefan Mendez-Diez, and Jonathan Rosenberg.
\newblock T-duality for orientifolds and twisted {$KR$}-theory.
\newblock {\em Lett. Math. Phys.}, 104(11):1333--1364, 2014, arXiv:1306.1779.

\bibitem{MR3316647}
Charles Doran, Stefan Mendez-Diez, and Jonathan Rosenberg.
\newblock String theory on elliptic curve orientifolds and {$KR$}-theory.
\newblock {\em Comm. Math. Phys.}, 335(2):955--1001, 2015.

\bibitem{Gao:2010ava}
Dongfeng Gao and Kentaro Hori.
\newblock {On the structure of the {C}han-{P}aton factors for {D}-branes in
  type {II} orientifolds}.
\newblock preprint, 2010, arXiv:1004.3972.

\bibitem{MR929766}
Daniel~R. Grayson.
\newblock The {$K$}-theory of semilinear endomorphisms.
\newblock {\em J. Algebra}, 113(2):358--372, 1988.

\bibitem{MR0164347}
Paul~S. Green.
\newblock A cohomology theory based upon self-conjugacies of complex vector
  bundles.
\newblock {\em Bull. Amer. Math. Soc.}, 70:522--524, 1964.

\bibitem{MR631748}
Benedict~H. Gross and Joe Harris.
\newblock Real algebraic curves.
\newblock {\em Ann. Sci. \'Ecole Norm. Sup. (4)}, 14(2):157--182, 1981.

\bibitem{Gukov:1999}
Sergei Gukov.
\newblock {{$K$}-theory, reality, and orientifolds}.
\newblock {\em Comm. Math. Phys.}, 210:621--639, 2000, arXiv:hep-th/9901042.

\bibitem{Hori:1999me}
Kentaro Hori.
\newblock {D-branes, T duality, and index theory}.
\newblock {\em Adv. Theor. Math. Phys.}, 3:281--342, 1999, hep-th/9902102.

\bibitem{MR0429992}
Teruo Kanzaki.
\newblock Note on quaternion algebras over a commutative ring.
\newblock {\em Osaka J. Math.}, 13(3):503--512, 1976.

\bibitem{MR0488029}
Max Karoubi.
\newblock {\em {$K$}-theory: An introduction}.
\newblock Springer-Verlag, Berlin, 1978.
\newblock Grundlehren der Mathematischen Wissenschaften, Band 226.

\bibitem{MR1958527}
Max Karoubi and Charles Weibel.
\newblock Algebraic and {R}eal {$K$}-theory of real varieties.
\newblock {\em Topology}, 42(4):715--742, 2003, arXiv:math/0509412.

\bibitem{MR1031992}
H.~Blaine Lawson, Jr. and Marie-Louise Michelsohn.
\newblock {\em Spin geometry}, volume~38 of {\em Princeton Mathematical
  Series}.
\newblock Princeton University Press, Princeton, NJ, 1989.

\bibitem{2011arXiv1110.6836M}
E.-K.~M. {Moutuou}.
\newblock Twistings of {$KR$} for {R}eal groupoids.
\newblock preprint, 2011, arXiv:1110.6836.

\bibitem{MoutuouThesis}
E.-K.~M. {Moutuou}.
\newblock {\em Twisted groupoid {$KR$}-theory}.
\newblock PhD thesis, Universit{\'e} de {L}orraine, 2012.
\newblock available at {\texttt{http://www.theses.fr/2012LORR0042}}.

\bibitem{2012arXiv1202.2057M}
E.-K.~M. {Moutuou}.
\newblock Graded {B}rauer groups of a groupoid with involution.
\newblock {\em J. Funct. Anal.}, 266(5):2689--2739, 2014, arXiv:1202.2057.

\bibitem{MR607081}
Shigeru Mukai.
\newblock Duality between {$D(X)$} and {$D(\widehat X)$} with its application
  to {P}icard sheaves.
\newblock {\em Nagoya Math. J.}, 81:153--175, 1981.

\bibitem{zbMATH05358709}
David Mumford.
\newblock {\em Abelian varieties}, volume~5 of {\em Tata Institute of
  Fundamental Research Studies in Mathematics}.
\newblock Published for the Tata Institute of Fundamental Research, Bombay, by
  Hindustan Book Agency, New Delhi, 2008.
\newblock With appendices by C. P. Ramanujam and Yuri Manin, Corrected reprint
  of the second (1974) edition.

\bibitem{Olver:2010:NHMF}
F.~W.~J. Olver, D.~W. Lozier, R.~F. Boisvert, and C.~W. Clark, editors.
\newblock {\em {NIST Handbook of Mathematical Functions}}.
\newblock Cambridge University Press, New York, NY, 2010.
\newblock Print companion to \cite{NIST:DLMF}.

\bibitem{MR1921811}
D.~O. Orlov.
\newblock Derived categories of coherent sheaves on abelian varieties and
  equivalences between them.
\newblock {\em Izv. Ross. Akad. Nauk Ser. Mat.}, 66(3):131--158, 2002.
\newblock translation in \emph{Izv. Math.}, 66(3):569--594, 2002.

\bibitem{MR1998775}
D.~O. Orlov.
\newblock Derived categories of coherent sheaves and equivalences between them.
\newblock {\em Uspekhi Mat. Nauk}, 58(3(351)):89--172, 2003.
\newblock translation in \emph{Russian Math. Surv.}, 58(3):511--591, 2003.

\bibitem{MR0447339}
S.~Parimala and R.~Sridharan.
\newblock Projective modules over quaternion algebras.
\newblock {\em J. Pure Appl. Algebra}, 9(2):181--193, 1976/77.

\bibitem{MR1219265}
C.~Pedrini and C.~Weibel.
\newblock Invariants of real curves.
\newblock {\em Rend. Sem. Mat. Univ. Politec. Torino}, 49(2):139--173 (1993),
  1991.

\bibitem{MR1633036}
Alexander Polishchuk and Eric Zaslow.
\newblock Categorical mirror symmetry: the elliptic curve.
\newblock {\em Adv. Theor. Math. Phys.}, 2(2):443--470, 1998.

\bibitem{MR0338129}
Daniel Quillen.
\newblock Higher algebraic {$K$}-theory. {I}.
\newblock In {\em Algebraic {$K$}-theory, {I}: {H}igher {$K$}-theories ({P}roc.
  {C}onf., {B}attelle {M}emorial {I}nst., {S}eattle, {W}ash., 1972)}, pages
  85--147. Lecture Notes in Math., Vol. 341. Springer, Berlin, 1973.

\bibitem{MR1002456}
Jeremy Rickard.
\newblock Morita theory for derived categories.
\newblock {\em J. London Math. Soc. (2)}, 39(3):436--456, 1989.

\bibitem{MR842428}
Jonathan Rosenberg.
\newblock {$C^\ast$}-algebras, positive scalar curvature, and the {N}ovikov
  conjecture. {III}.
\newblock {\em Topology}, 25(3):319--336, 1986.

\bibitem{MR1282290}
Jonathan Rosenberg.
\newblock {\em Algebraic {$K$}-theory and its applications}, volume 147 of {\em
  Graduate Texts in Mathematics}.
\newblock Springer-Verlag, New York, 1994.

\bibitem{MR3305978}
Jonathan Rosenberg.
\newblock Real {B}aum-{C}onnes assembly and {T}-duality for torus orientifolds.
\newblock {\em J. Geom. Phys.}, 89:24--31, 2015.

\bibitem{MR2077669}
Thomas Schick.
\newblock Real versus complex {$K$}-theory using {K}asparov's bivariant
  {$KK$}-theory.
\newblock {\em Algebr. Geom. Topol.}, 4:333--346, 2004.

\bibitem{MR772065}
Andrei~A. Suslin.
\newblock On the {$K$}-theory of local fields.
\newblock {\em J. Pure Appl. Algebra}, 34(2--3):301--318, 1984.
\newblock Proceedings of the {L}uminy conference on algebraic {$K$}-theory
  ({L}uminy, 1983).

\bibitem{MR799254}
Richard~G. Swan.
\newblock {$K$}-theory of quadric hypersurfaces.
\newblock {\em Ann. of Math. (2)}, 122(1):113--153, 1985.

\bibitem{MR570179}
George Szeto.
\newblock On generalized quaternion algebras.
\newblock {\em Internat. J. Math. Math. Sci.}, 3(2):237--245, 1980.

\bibitem{MR3076731}
Charles~A. Weibel.
\newblock {\em The {$K$}-{B}ook: An introduction to algebraic $K$-theory},
  volume 145 of {\em Graduate Studies in Mathematics}.
\newblock American Mathematical Society, Providence, RI, 2013.

\bibitem{MR1581415}
Ernst Witt.
\newblock Zerlegung reeller algebraischer {F}unktionen in {Q}uadrate.
  {S}chiefk\"orper \"uber reellem {F}unktionenk\"orper.
\newblock {\em J. Reine Angew. Math.}, 171:4--11, 1934.

\bibitem{Witten:1998}
Edward Witten.
\newblock D-branes and {$K$}-theory.
\newblock {\em J. High Energy Phys.}, 1998(12):Paper 19, 1998,
  arXiv:hep-th/9810188.

\bibitem{Witten:1998-02}
Edward Witten.
\newblock Toroidal compactification without vector structure.
\newblock {\em J. High Energy Phys.}, 1998(2):Paper 6, 1998,
  arXiv:hep-th/9712028.

\bibitem{MR1325783}
Dongyuan Yao.
\newblock A note on the {$K$}-theory of twisted projective lines and twisted
  {L}aurent polynomial rings.
\newblock {\em J. Algebra}, 173(2):424--435, 1995.

\end{thebibliography}
\end{document}